\documentclass[12pt]{amsart}

\usepackage{amssymb}
\usepackage{pdfsync}
\usepackage{cite}
\usepackage{slashbox}	
\usepackage{amsfonts}
\usepackage{cancel}	
\usepackage{amstext}
\usepackage{amsthm,amsmath,amsfonts,amssymb,amscd,latexsym,txfonts,stmaryrd}
\usepackage{color}
\usepackage{epsfig}
\usepackage[all,cmtip]{xy}
\usepackage{enumerate}

\newtheorem{theorem}{Theorem}[section]
\newtheorem*{theoremd*}{Definition/Theorem}
\newtheorem*{theorem1*}{Theorem 1}
\newtheorem*{theorem*}{Theorem 1}
\newtheorem*{problem*}{Problem}

\newtheorem*{question*}{Question}

\newtheorem*{remarks*}{Remarks}
\newtheorem*{claim*}{Claim}

\newtheorem*{remark*}{Remark}
\newtheorem{remark}{Remark}[section]

\newtheorem*{hLt*}{Hard Lefschetz Theorem}
\newtheorem*{HRR*}{Hodge-Riemann Bilinear Relations}
\newtheorem*{basisthm*}{Basis Theorem}
\newtheorem*{relbasisthm*}{Relative Basis Theorem}
\newtheorem*{primdecomp*}{Primitive Decomposition Theorem}
\newtheorem*{PD*}{Poincar\'e Duality}
\newtheorem*{JBC*}{Watanabe's Bold Conjecture}
\newtheorem{proposition}{Proposition}[section]
\newtheorem{lemma}{Lemma}[section]

\newtheorem*{corollary1*}{Corollary 1}
\newtheorem*{corollary*}{Corollary}
\newtheorem*{uptp}{Universal Property of Tensor Products}

\newcommand{\mSpec}{\operatorname{mSpec}}

\newcommand{\GL}{\operatorname{GL}}

\newcommand{\Maps}{\operatorname{Maps}}

\newcommand{\Z}{{\mathbb Z}}

\newcommand{\N}{{\mathbb N}}

\renewcommand{\P}{{\mathbb P}}
\newcommand{\F}{{\mathbb F}}

\numberwithin{equation}{section}

\begin{document}

\title[Equivariant Coinvariant Ring]{A GKM Description of the Equivariant Coinvariant Ring of a Pseudo-Reflection Group}
\author{Chris McDaniel}
\address{Dept. of Math. and Comp. Sci.\\
Endicott College\\
Beverly, MA 01915}
\email{cmcdanie@endicott.edu}






\begin{abstract} 
We identify the equivariant coinvariant ring of a pseudo-reflection group with its image under the localization map.  We then show that this image can be realized as the equivariant cohomology of a sort linear hypergraph, analogous to a GKM 1-skeleton.
\end{abstract}
\maketitle



\section{Introduction}
Let $\F$ be an algebraically closed field, $V=\F^n$ a finite dimensional vector space over that field, and let $R=\F[V]$ be the ring of polynomial functions on $V$.  Let $W\subset\GL(V)$ be any finite group with $|W|\in\F^\times$.  Then $W$ acts on $R$ by $w\cdot f(v)=f(w^{-1}(v))$.  Let $R^W\subset R$ denote the graded subring of $W$-invariant polynomials, and let $J_W\coloneqq \left(R^W\right)^+\cdot R$ denote the ideal in $R$ generated by the $W$-invariants of strictly positive degree.  Let $R_W\coloneqq R/J_W$ denote the \emph{coinvariant ring} of $W$.  The \emph{equivariant coinvariant ring} of $W$ is the ring $R\otimes_{R^W}R$.  A theorem of Chevalley and Shephard-Todd states that $R^W$ is a polynomial ring (which implies that $R$ is a free module over $R^W$), if and only if $W$ is generated by pseudo-reflections.  It follows that the equivariant coinvariant ring of a pseudo-reflection group is free as a left $R$ module.   

For each $x\in W$ define the twisted multiplication map $\mu_x\colon R\otimes_{R^W}R\rightarrow R$ by $\mu_x(f\otimes g)\coloneqq f\cdot x(g)$.  Putting these maps together yields the \emph{localization map}
\begin{equation}
\label{eq:Localization}
\xymatrixrowsep{1pc}\xymatrix{\mu\colon R\otimes_{R^W}R\ar[r] & \bigoplus_{x\in W}R\\
F\ar@{|->}[r] & \left(\mu_x(F)\right)_{x\in W}\\}
\end{equation}
It will be convenient to identify the ring $\bigoplus_{x\in W}R$ with the ring $\operatorname{Maps}(W,R)$ consisting of all maps from $W$ into $R$, so that a tuple $\left(F_x\right)_{x\in W}$ is identified with the map $\left\{x\mapsto F_x\right\}$.  

A \emph{co-root} associated to a pseudo-reflection $s\in W$ is any non-zero linear function $\ell_s\in V^*$ which vanishes on the $s$-fixed hyperplane $H_s\subset V$.  A co-root also has an associated map $L_s\colon W\rightarrow R$ defined by $L_s(x)=x(\ell_s)$.  Let $s(W)\subset W$ denote the set of all pseudo-reflections in $W$ and for any group element $x\in W$ let $|x|\in\N$ denote its order.

Here is our main result:

\begin{theorem}
\label{thm:Main}
The localization map is injective, and its image is equal to the subset 
$$\mathcal{H}_W=\left\{F\colon W\rightarrow R\left| \sum_{j=0}^{|s|-1}\frac{F(x\cdot s^j)}{L^i_s(x\cdot s^j)}\in R \ \ \forall \ x\in W, \ \ \forall \ s\in s(W), \ \ \forall \ i\leq |s|-1\right.\right\}.$$
\end{theorem}

The proof of Theorem \ref{thm:Main} comes in two parts:  Injectivity of the localization map and identification of the image.  We derive injectivity of the localization map by identifying the equivariant coinvariant ring $R\otimes_{R^W}R$ with the ring of regular functions on the union of graphs, i.e.
$$\mSpec(R\otimes_{R^W}R)\cong \bigcup_{x\in W}\left\{(M,x^{-1}(M))\left|M\in \mSpec(R)\right.\right\}.$$ 
While this fact has been pointed out by others before, e.g. \cite{S1,W0}, we give a careful proof of it here, and this is the bulk of the work we do in Part I. 

In Part II we identify the image of the localization map as $\mathcal{H}_W$.  To this end, we define operators ${_iA_s}\colon \Maps(W,R)\rightarrow\Maps(W,\operatorname{Quot}(R))$ by
\begin{equation}
\label{eq:iAs}
{_iA_s}(F)(x)=\sum_{j=0}^{|s|-1}\frac{F(x\cdot s^j)}{L_s^i(x\cdot s^j)}
\end{equation}
so that $\mathcal{H}_W$ consists of the maps $F\in \Maps(W,R)$ for which ${_iA_s}(F)\in\Maps(W,R)$ too.  These  
operators are equivariant versions of the so-called \emph{generalized $\Delta$ operators} ${_i\Delta_s}\colon R_W\rightarrow R_W(-i)$ introduced by McDaniel and Smith in their (forthcoming) paper \cite{SM}.  From this observation we deduce containment in one direction, namely 
$$\mu(R\otimes_{R^W}R)\subseteq\mathcal{H}_W.$$  
We show that the set $\mathcal{H}_W$ is closed under the operators ${_iA_s}$.  We then use these ``restricted'' operators ${_iA_s}\colon\mathcal{H}_W\rightarrow\mathcal{H}_W(-i)$ to deduce containment in the other direction.  

If $W$ is a Weyl group associated to a compact semi-simple Lie group $G$ and a Cartan subgroup $T\subset G$, then GKM theory computes the $T$-equivariant cohomology of the homogeneous space $X=G/T$ from the moment graph of the (left) $T$-action on $X$ which, in this case, agrees with our set $\mathcal{H}_W$.  On the other hand, the Borel description of the equivariant cohomology of $X$ identifies it as the equivariant coinvariant ring of $W$, $R\otimes_{R^W}R$, and reconciling these two distinct points of view recovers the isomorphism from Theorem \ref{thm:Main}:
$$\mathcal{H}_W\cong R\otimes_{R^W}R.$$
See the paper by Guillemin, Holm, and Zara \cite{GHZ} for further details and references.

In a series of papers \cite{GZ0,GZ1,GZ2} Guillemin and Zara introduced a combinatorial analogue of the moment graph of a GKM $T$-manifold, consisting of a regular graph $\Gamma=(\mathcal{V},\mathcal{E})$ and a linear function $\alpha\colon\mathcal{E}\rightarrow \P(V^*)$ assigning a linear subspace in $V^*$ to every edge of $\Gamma$ satisfying certain compatibility conditions.  They called the pair $(\Gamma,\alpha)$ a GKM 1-skeleton, and defined its \emph{equivariant cohomology} to be the subset of polynomial maps on the vertex set $\mathcal{V}$ such that on adjacent vertices, the polynomials agree along the annihilator of the linear subspace assigned to that edge, i.e.
\begin{equation}
\label{eq:EqCo}
H(\Gamma,\alpha)=\left\{F\colon\mathcal{V}\rightarrow R\left|F(q)-F(p)=c_{pq}\cdot\alpha(pq), \ \forall \ pq\in\mathcal{E}, \ \text{some $c_{pq}\in R$}\right.\right\}.
\end{equation}   

If $W$ is an arbitrary Coxeter group, there is an associated graph $\Gamma_W=(\mathcal{V}_W,\mathcal{E}_W)$ with vertex set $\mathcal{V}_W$ given by the elements of the group $W$, and edge set $\mathcal{E}_W$ given by (right) reflection orbits $\mathfrak{o}_s(x)\coloneqq \left\{x,x\cdot s\right\}$.  Further, an axial function on $\Gamma_W$ is given by the function $\alpha(\mathfrak{o}_s(x))=\F\cdot x(\ell_s)$.  In this case it is straight forward to see that equivariant cohomology of the resulting GKM 1-skeleton, $(\Gamma_W,\alpha_W)$ is equal to our set $\mathcal{H}_W$.  

If $W$ is an arbitrary pseudo-reflection group, then the reflection orbits $\mathfrak{o}_s(x)$ may have more than two vertices, making $\Gamma_W$ into a \emph{hypergraph}.  Moreover the equivariant cohomology of the object $(\Gamma_W,\alpha_W)$ as defined in Equation \eqref{eq:EqCo} is no longer isomorphic to the equivariant coinvariant ring $R\otimes_{R^W}R$.  For example consider the cyclic group $W=\langle s\rangle$ generated by a single pseudo-reflection of order $d>2$.  In this case the equivariant coinvariant ring is generated by powers of a single linear element $\left\{1\otimes\ell_s^i\left|0\leq i\leq |s|-1\right.\right\}$ while generators of the module $H(\Gamma_W,\alpha_W)$ given in Equation \eqref{eq:EqCo} all have degree one.  

Following Guillemin and Zara \cite{GZ2}, we give an alternative definition of the equivariant cohomology of a hyperedge $e\in\mathcal{E}$ by first specifying a \emph{generating class} $\tau_e\colon\mathcal{V}_e\rightarrow \alpha(e)$ then taking the free $R$-submodule of $\Maps(\mathcal{V}_e,R)$ generated by powers of the generating class $\left\{1,\tau,\ldots,\tau^{|e|-1}\right\}$, i.e.
$$H(e,\tau_e)=\left\{G\colon\mathcal{V}_e\rightarrow R\left|G=\sum_{j=0}^{|e|-1}g_i\tau_e^i, \ \text{some constants $g_i\in R$}\right.\right\}.$$
If $\tau=\left\{\tau_e\right\}_{e\in\mathcal{E}}$ is a compatible system of generating classes for $(\Gamma,\alpha)$ we say that the triple $(\Gamma,\alpha,\tau)$ is a \emph{linear hypergraph} and we define its equivariant cohomology to be those polynomial maps on $\mathcal{V}$ whose restriction to each hyperedge gives an element in the equivariant cohomology of that hyperedge, i.e.
\begin{equation}
\label{eq:HyperEqCo}
H(\Gamma,\alpha,\tau)=\left\{F\colon\mathcal{V}\rightarrow R\left|\rho_e(F)\in H(e,\tau_e), \ \forall \ e\in\mathcal{E}\right.\right\}.
\end{equation}
We show for any finite pseudo-reflection group $W$, there is a natural choice for a compatible system of generating classes $\tau_W$ for $(\Gamma_W,\alpha_W)$ such that the equivariant cohomology $H(\Gamma_W,\alpha_W,\tau_W)$ as defined in Equation \eqref{eq:HyperEqCo} actually does coincide with our set $\mathcal{H}_W$ from Theorem \ref{thm:Main}.  The resulting isomorphism from Theorem \ref{thm:Main} and the above identification, i.e. 
\begin{equation}
\label{eq:GKMCoinvariant}
H(\Gamma_W,\alpha_W,\tau_W)\cong R\otimes_{R^W}R
\end{equation}
then gives us our so-called GKM description of the equivariant coinvariant ring of the pseudo-reflection group $W$.

This paper is organized as follows.  In Section 2, we prove injectivity of the localization map.  In Section 3, we identify the image of the localization map with our subset $\mathcal{H}_W$.  In Section 4, we define linear hypergraphs and their equivariant cohomology, we describe the linear hypergraph associated to a psuedo-reflection group $W$, and we identify its equivariant cohomology with our set $\mathcal{H}_W$.

\section{Part One:  Injectivity of the Localization Map}
Let $\F$, $V$, $R$, $W$ be as above.  If $S\subset R$ is any $\F$ subalgebra, we can form the tensor product $R\otimes_SR$.  The tensor product comes with ``factor maps'' $\pi_{S,1},\pi_{S,2}\colon R\rightarrow R\otimes_SR$ defined by $\pi_{S,1}(r)=r\otimes 1$ and $\pi_{S,2}(r)=1\otimes r$.  The tensor product and its factor maps satisfy the following universal property:
\begin{uptp}
If $A$ is any $\F$ algebra, and if there exist maps of $\F$ algebras $\phi_1,\phi_2\colon R\rightarrow A$ such that $\phi_1|_S=\phi_2|_S$, then there exists a unique map of $\F$ algebras $\Phi\colon R\otimes_SR\rightarrow A$ making the following diagram commute:
$$\xymatrix{& R\otimes_SR\ar@{-->}[rr]^-{\Phi} & & A & \\
R\ar[ru]^-{\pi_{S,1}}\ar[rrru]^-{\phi_1} & & & & R\ar[lllu]_-{\pi_{S,2}}\ar[lu]_-{\phi_2}\\
& & S\ar[llu]\ar[rru] & &\\}$$ 
\end{uptp}
In particular, note that if $S'\subseteq S$ is any $\F$ subalgebra, there is a unique surjective map of $\F$ algebras $q_{S',S}\colon R\otimes_{S'}R\rightarrow R\otimes_{S}R$ defined by $q_{S',S}(f\otimes g)=f\otimes g$.  In case $S'=\F$, we will denote this map $q_S\colon R\otimes R\rightarrow R\otimes_S R$, where by $R\otimes R$ we mean $R\otimes_{\F}R$.  In the special case where $S=R^W$, we write $q_W\colon R\otimes R\rightarrow R\otimes_{R^W}R$ and we set $I_W\coloneqq \ker(q_W)\subset R\otimes R$.

Note that in the universal property, we may also take $A=R$ and $\phi_{i}=\operatorname{id}_R$ to get the multiplication map $\hat{\mu}\colon R\otimes R\rightarrow R$, $\hat{\mu}(f\otimes g)=f\cdot g$.  More generally, if $x\colon R\rightarrow R$ is an $\F$-algebra automorphism, we can define the $x$-twisted multiplication map $\hat{\mu}_x\colon R\otimes R\rightarrow R$, $\hat{\mu}_x(f\otimes g)=f\cdot x(g)$.  Let $I_x\subset R\otimes R$ be the kernel of $\hat{\mu}_x$.  If that automorphism $x\in\operatorname{Aut}(R)$ fixes the subring $S$, i.e. $x|_S=\operatorname{id}_S$, then we get $x$-twisted $S$-multiplication maps $\mu_{S,x}\colon R\otimes_SR\rightarrow R$, $\mu_{S,x}(f\otimes g)=f\cdot x(g)$.  In particular we have $\mu_{S,x}\circ q_S=\hat{\mu}_x$.  In the special case where $S=R^W$, we simply write the $x$-twisted $S$-multiplication map as $\mu_x\colon R\otimes_{R^W}R\rightarrow R$.  

Define the localization map $\mu\colon R\otimes_{R^W}R\rightarrow\bigoplus_{x\in W}R$ and the lifted localization map $\hat{\mu}\colon R\otimes R\rightarrow \bigoplus_{x\in W}R$ by $\mu(F)=\left(\mu_x(F)\right)_{x\in W}$ and $\hat{\mu}(F)=\left(\hat{\mu}_x(F)\right)_{x\in W}$, respectively.  Then clearly we have 
$$\mu\circ q_W=\hat{\mu},$$
and hence in order to see that the localization map is injective, we need to show that $\ker(q_W)=\ker(\hat{\mu})$.  Note that $\ker(\hat{\mu})=\bigcap_{x\in W}\ker(\hat{\mu}_x)=\bigcap_{x\in W}I_x$ and $\ker(q_W)=I_W$.  Thus we need to show that 
\begin{equation}
\label{eq:WTS}
I_W=\bigcap_{x\in W}I_x.
\end{equation}

We can prove Equation \eqref{eq:WTS} using some basic facts from commutative algebra.  For a commutative ring $A$ define $\mSpec(A)$ to be the set of maximal ideals or the \emph{maximal spectrum of $A$}.  For $I\subseteq A$ an ideal define the subset $\mathcal{V}_m(I)\subset\mSpec(A)$ to be the set of maximal ideals in $A$ containing $I$.  Note there is a bijection between the sets $\mSpec(A/I)$ and $\mathcal{V}_m(I)$.  We give some basic properties of the operator $\mathcal{V}_m$:
\begin{lemma}
\label{lem:VmProp}
Let $A$ be a commutative ring, and let $I,J\subseteq A$ be ideals
\begin{enumerate}
\item If $I\subseteq J$ then $\mathcal{V}_m(I)\supseteq\mathcal{V}_m(J)$.
\item We have $\mathcal{V}_m(I\cap J)=\mathcal{V}_m(I)\cup\mathcal{V}_m(J)$.
\item If $A$ is a finitely generated $\F$ algebra, then $\mathcal{V}_m(I)=\mathcal{V}_m(J)$ if and only if $\sqrt{I}=\sqrt{J}$.
\end{enumerate}
\end{lemma}
\begin{proof}
If $I\subseteq J$, and $M\in\mSpec(A)$ is a maximal ideal containing $J$, then certainly $M$ must also contain $I$.  Hence $\mathcal{V}_m(J)\subseteq\mathcal{V}_m(I)$.  If $M\in\mSpec(A)$ is a maximal ideal containing either $I$ or $J$, then certainly $M$ must also contain $I\cap J$.  Conversely if $M$ contains $I\cap J$ and $M\not\supseteq J$ take $b\in J\setminus M$.  Then for each $a\in I$, we have $a\cdot b\in I\cap J\subseteq M$ and since $M$ is prime and $b\notin M$ we must therefore have $a\in M$.  Since this holds for each $a\in I$, we conclude that $I\subset M$.  This shows that $\mathcal{V}_m(I\cap J)=\mathcal{V}_m(I)\cup\mathcal{V}_m(J)$.  Finally, if $A$ is a finitely generated $\F$-algebra recall that the radical of an ideal is the intersection of \emph{maximal} ideals containing it \cite[Theorem 5.5]{Mat}.  Now suppose that $\mathcal{V}_m(I)=\mathcal{V}_m(J)$.  Then certainly we have the equality
$$\sqrt{I}=\bigcap_{M\in\mathcal{V}_m(I)}M=\bigcap_{M\in\mathcal{V}_m(J)}M=\sqrt{J}.$$
Conversely suppose that $\sqrt{I}=\sqrt{J}$ and let $M\in\mathcal{V}_m(I)$.  If $b\in J\subseteq\sqrt{J}=\sqrt{I}$ then $b^N\in I\subseteq M$ for some $N>0$.  But $M$ is prime thus $b\in M$ and this holds for all $b\in J$ hence $J\subseteq M$.  This shows $\mathcal{V}_m(I)\subseteq\mathcal{V}_m(J)$ and we may repeat the argument replacing $I$ with $J$ to find that $\mathcal{V}_m(J)\subseteq\mathcal{V}_m(J)$as well.
\end{proof} 

Recall that if $\phi\colon A\rightarrow B$ is a map of finitely generated $\F$ algebras and $\F$ is algebraically closed, then $\phi$ induces a map $\phi^*\colon\mSpec(B)\rightarrow\mSpec(A)$, $\phi^*(M)=\phi^{-1}(M)$ (this is a direct consequence of Zariski's Lemma \cite[Theorem 5.3]{Mat}).  

Let $R$ be the polynomial ring as above.
Let $\pi_i\colon R\rightarrow R\otimes R$ be the $i^{th}$ factor map as above, and consider the map
\begin{equation}
\label{eq:Prod}
(\pi_1^*,\pi_2^*)\colon \mSpec(R\otimes R)\rightarrow\mSpec(R)\times\mSpec(R)
\end{equation}
We claim that the map in Equation \eqref{eq:Prod} is a bijection.  Despite its intuitive feel, this fact is not quite obvious.  To prove it, we need the following lemma.
\begin{lemma}
\label{lem:FactorSpec}
If $M_1,M_2\in\mSpec(R)$ then the ideal $M=M_1\otimes R+R\otimes M_2\subseteq R\otimes R$ is the kernel of the map
$$\phi\colon R\otimes R\rightarrow R/M_1\otimes R/M_2\cong \F.$$
In particular, $M$ is a maximal ideal.  Moreover we have $\pi_1^{-1}(M)=M_1$ and $\pi_2^{-1}(M)=M_2$.
\end{lemma}
\begin{proof}
Certainly $M\subset\ker(\phi)$.  By way of contradiction, assume that $\ker(\phi)\not\subseteq M$ and suppose that $x\in\ker(\phi)\setminus M$.  We may write $x$ as a sum of simple tensors $x=\sum_{i=1}^Kf_i\otimes g_i$ for $f_i,g_i\in R$ and some positive integer $K$.  To get the contradiction we assume that we have chosen our $x$ with $K$ as small as possible.  Note that $\phi\colon R\otimes R\rightarrow R/M_1\otimes R/M_2$ factors through the maps $\phi_1\colon R\otimes R\rightarrow R\otimes R/M_2$ and $\phi_2\colon R\otimes R/M_2\rightarrow R/M_1\otimes R/M_2$.  We have 
$$\phi_1(x)=\sum_{i=1}^Kf_i\otimes\bar{g}_i=\sum_{i=1}^Kf_i\cdot\bar{g}_i\otimes 1.$$
Now since $\phi_2(\phi_1(x))=\phi(x)=0$ the sum $\sum_{i=1}^Kf_i\cdot\bar{g}_i$ must lie in the maximal ideal $M_1$.  If $\bar{g}_i=0$ for all $i$ then $g_i\in M_2$ for all $i$, and $x\in R\otimes M_2\subset M$, contrary to our choice of $x$.  On the other hand, if $\bar{g}_j\neq 0$ for some index $1\leq j\leq K$ then we may eliminate the index $j$ from our sum and replace it by an element of $M$, i.e.  
$$x=\underbrace{\sum_{\substack{i=1\\ i\neq j\\}}^{K} f_i\otimes \left(g_i-\frac{\bar{g}_i}{\bar{g}_j}\cdot g_j\right)}_{x'} +m$$
for some $m\in M$.  But now we have found another element $x'\in\ker(\phi)\setminus M$ that can be represented as a sum of fewer simple tensors than $x$, again contrary to our choice of $x$.
Thus we must have $\ker(\phi)\subset M$.
\end{proof}

\begin{lemma}
\label{lem:Product}
The map in Equation \eqref{eq:Prod} is bijective.
\end{lemma}
\begin{proof}
Lemma \ref{lem:FactorSpec} implies that the map is surjective.  To see that it is injective, fix maximal ideals $M_1,M_2\in\mSpec(R)$ and $M=M_1\otimes R+R\otimes M_2\in\mSpec(R\otimes R)$ and suppose that $\hat{M}\in\mSpec(R\otimes R)$ is another maximal ideal such that $\pi_1^{-1}(\hat{M})=M_1$ and $\pi_2^{-1}(\hat{M})=M_2$.  Then certainly we have the containment $M\subseteq \hat{M}$.  But since $M$ and $\hat{M}$ are both maximal ideals, the containment must be an equality.
\end{proof}

Let $S\subset R$ be any $\F$ subalgebra as above with $q_S\colon R\otimes R\rightarrow R\otimes_SR$.  We want to understand the subset $q_S^*\left(\mSpec(R\otimes_SR)\right)\subset\mSpec(R\otimes R)$, or, equivalently, the subset $(\pi_1^*,\pi_2^*)\circ q_S\left(\mSpec(R\otimes_SR)\right)=\left(\pi_{S,1}^*,\pi_{S,2}^*\right)\left(\mSpec(R\otimes_SR)\right)$.  
\begin{lemma}
\label{lem:mSpec}
The image $(\pi_{S,1}^*,\pi_{S,2}^*)\left(\mSpec(R\otimes_SR)\right)$ is equal to the set
\begin{equation}
\label{eq:RSR}
\mSpec(R)\times_S\mSpec(R)\coloneqq \left\{(M_1,M_2)\in\mSpec(R)\times\mSpec(R)\left|M_1\cap S=M_2\cap S\right.\right\}.
\end{equation}
\end{lemma}
\begin{proof}
Containment in one direction is easy:
$$(\pi_{S,1}^*,\pi_{S,2}^*)\left(\mSpec(R\otimes_SR)\right)\subseteq \mSpec(R)\times_S\mSpec(R).$$  
Indeed for any $M\in\mSpec(R\otimes_S R)$, and any $x\in\pi^{-1}_{S,1}(M)\cap S$, we have 
\begin{align*}
\pi_{S,1}(x)= & x\otimes 1\\
= & 1\otimes x\\
= & \pi_{S,2}(x)
\end{align*}
hence $x\in\pi^{-1}_{S,2}(M)\cap S$ as well, and this argument can be repeated, replacing $\pi_{S,1}$ with $\pi_{S,2}$.  Hence $\pi_1^{-1}(q^{-1}(M))\cap S=\pi_2^{-1}(q^{-1}(M))\cap S$.  

Conversely, fix $(M_1,M_2)\in\mSpec(R)\times\mSpec(R)$ such that $M_1\cap S=M_2\cap S$.  We must find a maximal ideal $M\subset R\otimes_SR$ such that $\pi_{S,i}^{-1}(M)=M_i$ for $i=1,2$.  Define the maps $\phi_1,\phi_2\colon R\rightarrow R/M_1\otimes R/M_2$ by $\phi_1(r)=\bar{r}\otimes 1$ and $\phi_2(r)=1\otimes \bar{\bar{r}}$ where $\bar{r}$ is reduction of $r$ modulo $M_1$ and $\bar{\bar{r}}$ is reduction of $r$ modulo $M_2$.  Note that $S\cap M_1$ is a maximal ideal in $S$ (by Lemma \ref{lem:Larry} below), hence the inclusion map $S\hookrightarrow R$ induces an isomorphism 
$S/S\cap M_1\rightarrow R/M_1$, and similarly for $M_2$.  Since $S\cap M_1=S\cap M_2$, we must therefore have $\bar{s}=\bar{\bar{s}}$ for each $s\in S$, which means that $\phi_1|_S=\phi_2|_S$.  Hence by the universal property for tensor products, there exists a unique map of $\F$ algebras $\Phi\colon R\otimes_SR\rightarrow R/M_1\otimes R/M_2$ such that $\Phi(f\otimes g)=\bar{f}\otimes\bar{\bar{g}}$.  Hence $M\coloneqq\ker(\Phi)$ is a maximal ideal in $R\otimes_SR$.  And since the map $\Phi$ clearly factors through the obvious map $\Phi'\colon R\otimes R\rightarrow R/M_1\otimes R/M_2$ via $q_S$, it follows that $q_S^{-1}(M)=M_1\otimes R+R\otimes M_2$, and the result follows from Lemma \ref{lem:FactorSpec}.
\end{proof}

For a proof of the following lemma we refer the reader to Smith's book on invariant theory \cite{Smith}, specifically Lemma 5.4.1 and Theorem 5.4.5.
\begin{lemma}
\label{lem:Larry}
Suppose that $A$ and $B$ are finitely generated $\F$ algebras, and $A\subset B$ with $B$ integral over $A$.  Further suppose that their respective fields of fractions $K_A\subset K_B$ is a Galois extension with Galois group $G$.  Then the Galois group $G$ acts transitively on the set of prime ideals in $B$ lying over a given prime ideal $P\subset A$.  (Also in this situation, given a prime ideal $P\subset B$ lying over a given prime ideal $Q\subset A$, $P$ is maximal if and only if $Q$ is maximal.)
\end{lemma}

Recall that $R^W$ and $R$ are both finitely generated $\F$ algebras, and that the extension $R^W\subset R$ is integral.  Moreover if $L^W$ is the field of fractions of $R^W$ and $L$ is the field of fractions of $R$ then $L^W\subset L$ is the field fixed by $W$ viewed as a group of field automorphisms of $L$.  Hence the extension $L^W\subset L$ is Galois with Galois group $W$.  Hence by Lemma \ref{lem:Larry} the group $W$ acts transitively on the set of maximal ideals in $R$ lying over a given maximal ideal in $S=R^W$.  Hence for any pair $M_1,M_2\in\mSpec(R)$ such that $M_1\cap R^W=M_2\cap R^W$ there exists $x\in W$ such that $M_2=x(M_1)$.  Replacing $S$ by $R^W$ in Lemma \ref{lem:mSpec} we thus obtain
\begin{align}
\label{eq:RWR}
\nonumber\left(\pi_1^*,\pi_2^*\right)\left(\mathcal{V}_m(I_W)\right)= & \mSpec(R)\times_{R^W}\mSpec(R)\\
\nonumber= & \left\{(M,x(M))\left|M\in\mSpec(R), \ x\in W\right.\right\}\\
= & \bigcup_{x\in W}\left\{(M,x(M))\left|M\in\mSpec(R)\right.\right\}.
\end{align}

Finally note that 
\begin{align}
\label{eq:MuxSpec}
\nonumber(\pi_1^*,\pi_2^*)\left(\mathcal{V}_m(I_x)\right)= & (\pi_1^*,\pi_2^*)\left(\hat{\mu}_x^*\left(\mSpec(R)\right)\right)\\
= & \left\{(M,x^{-1}(M))\left|M\in\mSpec(R)\right.\right\}
\end{align}
Indeed recall that $\hat{\mu}_x\circ\pi_1=\operatorname{id}_R$ and $\hat{\mu}_x\circ\pi_2=x$.  Thus given any maximal ideal $M\in \mSpec(R)$ we have $\pi_1^{-1}\circ\hat{\mu}_x^{-1}(M)=M$ and $\pi_2^{-1}\circ\hat{\mu}_x^{-1}(M)=x^{-1}(M)$.

Combining Equation \eqref{eq:RWR} and \eqref{eq:MuxSpec} we see that 
$$(\pi_1^*,\pi_2^*)\left(\mathcal{V}_m(I_W)\right)= \bigcup_{x\in W}(\pi_1^*,\pi_2^*)\left(\mathcal{V}_m(I_x)\right)$$
or, equivalently,
\begin{equation}
\label{eq:RegFcn}
\mathcal{V}_m(I_W)=\bigcup_{x\in W}\mathcal{V}_m(I_x).
\end{equation}
By Lemma \ref{lem:VmProp}, we deduce that 
$$\sqrt{I_W}=\sqrt{\bigcap_{x\in W}I_x}.$$
Note that $I_x$ is prime though, hence $\sqrt{\bigcap_{x\in W}I_x}=\bigcap_{x\in W}I_x$ since the intersection of prime ideals is always radical.  Therefore in order to prove that the localization map is injective, we need only show that $I_W$ is also radical, or equivalently, that $R\otimes_{R^W}R$ is reduced.  To wit:
\begin{lemma}
\label{lem:Reduced}
The ring $R\otimes_{R^W}R$ is reduced, i.e. it has no nilpotent elements.
\end{lemma}
\begin{proof}
Recall that $R$ is a finitely generated free module over the subalgebra $R^W\subset R$.  Choose and fix a basis, say $\left\{e_1,\ldots,e_N\right\}\subset R$.  Thus $R\otimes_{R^W}R$ is a finitely generated free left $R$ module with a basis $\left\{1\otimes e_1,\ldots,1\otimes e_N\right\}$.

Now suppose, by way of contradiction, that $R\otimes_{R^W}R$ is not reduced.  Then there is a non-zero nilpotent element, say $F\in R\otimes_{R^W}R$.  Since $R\otimes_{R^W}R$ is graded, we may assume that $F$ is homogeneous and for argument's sake, we may assume that it has minimal degree.

First note that if $F\in R\otimes_{R^W}R$ is nilpotent, then so is $\phi_{s,1}(F)$ for each $s\in W$, since $\phi_{s,1}$ is a ring homomorphism.  But then $F-\phi_{s,1}(F)$ and hence $\Delta_{s,1}(F)$ must also be nilpotent for each pseudo-reflection $s\in W$.  On the other hand, since the degree of $\Delta_{s,1}(F)$ is strictly less than the degree of $F$ and since we chose $F$ to have minimal degree, we must conclude that $\Delta_{s,1}(F)$ zero for each pseudo-reflection $s\in W$.  Hence we must have that $\phi_{s,1}(F)=F$ for all pseudo-reflections $s\in W$.

Now write $F$ in terms of our fixed basis above, i.e.
$$F=\sum_{i=1}^N f_i(1\otimes e_i)=\sum_{i=1}^Nf_i\otimes e_i, \ \ \ f_i\in R.$$
Note that since $\left\{1\otimes e_1,\ldots,1\otimes e_N\right\}$ are linearly independent, $\phi_{s,1}(F)=F$ implies that for $1\leq i\leq N$ we have $s(f_i)=f_i$.  But since this holds for each pseudo-reflection $s\in W$, and since $W$ is generated by its pseudo-reflections, we must conclude that $f_i\in R^W$ for each $1\leq i\leq N$.  Therefore we may write 
$$F=\sum_{i=1}^Nf_i\otimes e_i= 1\otimes \left(\sum_{i=1}^N f_i\cdot e_i\right).$$
This means that $F$ is in the image of the factor map $\pi_{W,2}\colon R\rightarrow R\otimes_{R^W}R$.  On the other hand, the factor map is injective--simply compose it with the usual multiplication map to get the identity map on $R$!  Hence if $F$ is nilpotent, then so is $\left(\sum_{i=1}^N f_i\cdot e_i\right)\in R$, and if $F$ is non-zero, then so is $\left(\sum_{i=1}^N f_i\cdot e_i\right)$.  But this is impossible since $R$ is reduced, and there is our contradiction!
\end{proof}

\begin{remark}
Lemma \ref{lem:Reduced} was also proved by J. Watanabe using a different argument \cite{W0}.
\end{remark} 

We have thus proved the following result.
\begin{proposition}
\label{prop:Injective}
The localization map $\mu\colon R\otimes_{R^W}R\rightarrow\bigoplus_{x\in W}R$ is injective.
\end{proposition}

\section{Part II:  The Image of the Localization Map}
We consider the set $\operatorname{Maps}(W,R)\cong\bigoplus_{x\in W}R$.  It is a graded ring with addition and multiplication defined pointwise, i.e. 
\begin{align*}
(F+G)(x)= & F(x)+G(x)\\
(F\cdot G)(x)= & F(x)\cdot G(x)
\end{align*}
We endow it with an $R$-module structure by taking the diagonal action, i.e.
$$(r\cdot F)(x)=r\cdot F(x).$$
In fact, regarding $R\subset\operatorname{Maps}(W,R)$ as the constant maps gives $\operatorname{Maps}(W,R)$ the structure of an $R$-algebra.

There is also right action of $W$ on the ring $\operatorname{Maps}(W,R)$ given by
\begin{equation}
\label{eq:HAction}
(F\cdot w)(x)\coloneqq F(x\cdot w^{-1}).
\end{equation}
The corresponding action of $W$ on $R\otimes_{R^W}R$, i.e. the one that makes the localization map $W$-equivariant, is given by
\begin{equation}
\label{eq:RAction}
(f\otimes g)\cdot w\coloneqq f\otimes w^{-1}(g).
\end{equation}

For each pseudo-reflection $s\in s(W)$ choose and fix a co-root $\ell_s\in V^*$.  This choice determines a map $L_s\colon W\rightarrow R$ defined by 
$$L_s(x)\coloneqq x\left(\ell_s\right).$$

Let $s\in s(W)$ be any pseudo-reflection, and let $i\in\Z$ be any integer.  Define the operator $_iA_s\colon \Maps(W,R)\rightarrow \Maps(W,Q)$ by 
\begin{equation}
\label{eq:Ais}
_iA_s(F)(x)=\sum_{j=0}^{|s|-1}\frac{F\cdot s^{-j}}{L_s\cdot s^{-j}}(x).
\end{equation}
Define the subset $\mathcal{H}_W\subset\Maps(W,R)$ by
$$\mathcal{H}_W=\left\{\left.F\colon W\rightarrow R\right| {_iA_s}(F)(x)\in R, \ \ \forall \ x\in W, \ \ \forall \ s\in s(W), \ \ \forall \ i\leq |s|-1\right\}.$$

The first thing we note is that $\mathcal{H}_W$ is an $R$-submodule of $\operatorname{Maps}(W,R)$, because the operator ${_iA_s}\colon\mathcal{H}_W\rightarrow \operatorname{Maps}(W,R)(-i)$ is an $R$-module map.  In fact we can say a bit more:
\begin{lemma}
\label{lem:One}
The constant map $1\colon W\rightarrow R$ assigning the value $1\in\F$ to each $x\in W$ is in the subset $\mathcal{H}_W$.
\end{lemma}
\begin{proof}
For $s\in s(W)$ and $i\leq |s|-1$ we have
\begin{align*}
_iA_s(1)(x)= & \sum_{j=0}^{|s|-1}\frac{1}{L_s^i(x\cdot s^j)}\\
= & \sum_{j=0}^{|s|-1}\frac{1}{\lambda_s^{ij}\cdot x(\ell_s)^i}\\
= & \frac{1}{x(\ell_s)^i}\sum_{j=0}^{|s|-1}\lambda_s^{-ij}\\
= & 0
\end{align*}
and the result follows.
\end{proof}
It follows that the subset $\mathcal{H}_W$ contains all constant maps.

We also note that the subset $\mathcal{H}_W$ is closed under the $W$-action. 
\begin{lemma}
\label{lem:WClosed}
If $F\in\mathcal{H}_W$ then so is $F\cdot w\in\mathcal{H}_W$ for any $w\in W$.
\end{lemma}
\begin{proof}
Fix $s\in s(W)$ and $i\leq |s|-1$.  For each $x\in W$ we have 
\begin{align*}
{_iA_s}(F\cdot w)(x)= & \sum_{j=0}^{|s|-1}\frac{(F\cdot w)(x\cdot s^j)}{L^i_s(x\cdot s^j)}\\
= & \sum_{j=0}^{|s|-1}\frac{F(x\cdot s^j\cdot w^{-1})}{x\cdot s^j\left(\ell_s\right)^i}\\
= & \sum_{j=0}^{|s|-1}\frac{F(x\cdot w^{-1}\cdot\left(w\cdot s^j\cdot w^{-1}\right))}{\lambda_s^{ij}x\cdot w^{-1}\left(w(\ell_s)\right)^i}\\
= & \sum_{j=0}^{|s|-1}\frac{F(x\cdot w^{-1}\cdot\left(w\cdot s^j\cdot w^{-1}\right))}{\lambda_s^{ij}x\cdot w^{-1}\left(c\cdot\ell_{wsw^{-1}}\right)^i}\\
= & \sum_{j=0}^{|s|-1}\frac{F(x\cdot w^{-1}\cdot\left(w\cdot s^j\cdot w^{-1}\right))}{c\cdot x\cdot w^{-1}\left(ws^jw^{-1}\right)\left(\ell_{wsw^{-1}}\right)^i}\\
= & \frac{1}{c}\cdot {_iA_{wsw^{-1}}}(F)
\end{align*}
which is clearly in $R$ since $F\in\mathcal{H}_W$.
\end{proof}

Note that for each pseudo-reflection $s\in s(W)$ and each integer $i\leq |s|-1$, the map ${_iA_s}(F)\colon W\rightarrow R$ is $s$-invariant.  Also, if $F\in\mathcal{H}_W$ is $s$-invariant already, then we have 
$$_iA_s(F)=F\cdot{_iA_s}(1).$$
The operators ${_iA_s}$ can be viewed as projection operators onto the $s$-invariant pieces of $\mathcal{H}_W$.
 
\begin{lemma}
\label{lem:isDecomp}
If $F\in\mathcal{H}_W$ then we have 
$$F=\frac{1}{|s|}\cdot\left(\sum_{j=0}^{|s|-1}{_jA_s}(F)\cdot L_s^j.\right)$$
\end{lemma}
\begin{proof}
We compute the RHS:
\begin{align*}
\sum_{j=0}^{|s|-1}{_jA_s}(F)\cdot L_s^j= & \sum_{j=0}^{|s|-1}\left(\sum_{k=0}^{|s|-1}\frac{F\cdot s^{-k}}{L_s^j\cdot s^{-k}}\right)\cdot L_s^j\\
= & \sum_{j=0}^{|s|-1}\sum_{k=0}^{|s|-1}\frac{F\cdot s^{-k}}{\lambda_s^{jk}\cdot L_s^j}\cdot L_s^j\\
= & \sum_{j=0}^{|s|-1}\sum_{k=0}^{|s|-1}\lambda_s^{-jk} F\cdot s^{-k}\\
= & \sum_{k=0}^{|s|-1}F\cdot s^{-k} \left(\sum_{j=0}^{|s|-1}\lambda_s^{-jk}\right) \\
= & F\cdot s^0\cdot\left(\sum_{j=0}^{|s|-1}\lambda_s^{j\cdot 0}\right)+\sum_{k=1}^{|s|-1}F\cdot s^{-k}\left(\sum_{j=0}^{|s|-1}\lambda_s^{jk}\right)
= & F\cdot |s|
\end{align*}
and the result follows.
\end{proof}

The following lemma is analogous to the usual Leibniz rule for ordinary $\Delta$ operators.
\begin{lemma}
\label{lem:GenLeib}
If $F,G\in \mathcal{H}_W$ then we have 
$${_iA_s}(F\cdot G)=\sum_{a=0}^{|s|-1}{_aA_s}(F)\cdot{_{i-a}A_s}(G).$$
\end{lemma}
\begin{proof}
By Lemma \ref{lem:isDecomp} we may write $F=\frac{1}{|s|}\sum_{a=0}^{|s|-1}{_aA_s}(F)\cdot L_s^a$.  Note that $_aA_s(F)$ is $s$-invariant.  Also note that $_iA_s(L_s^a\cdot G)={_{i-a}A_s}(G)$ for every $G$ and every $a$.  Putting it all together, we have 
\begin{align*}
{_iA_s}(F\cdot G)= & {_iA_s}\left(\frac{1}{|s|}\sum_{a=0}^{|s|-1}{_aA_s}(F)\cdot L_s^a\cdot G\right)\\
= & \frac{1}{|s|}\sum_{a=0}^{|s|-1}{_iA_s}\left({_aA_s}(F)\cdot L_s^a\cdot G\right)\\
= & \frac{1}{|s|}\sum_{a=0}^{|s|-1}{_aA_s}(F)\cdot {_iA_s}(L_s^a\cdot G)\\
= & \frac{1}{|s|}\sum_{a=0}^{|s|-1}{_aA_s}(F)\cdot {_{i-a}A_s}(G)
\end{align*}
as claimed.
\end{proof}
Note that an immediate consequence of Lemma \ref{lem:GenLeib} is that the subset $\mathcal{H}_W$ is closed under multiplication.  In particular, we see that the subset $\mathcal{H}_W\subset\operatorname{Maps}(W,R)$ is an $R$-subalgebra.  The next proposition is fundamental to our main results.

\begin{proposition}
\label{prop:AisClosed}
If $F\in\mathcal{H}_W$ then so is $_iA_s(F)\in \mathcal{H}_W$.
\end{proposition}
\begin{proof}
We need to show that for any other pseudo-reflection $t\in s(W)$ and any integer $j\leq |t|-1$, if $F\in\mathcal{H}_W$, then for every $x\in W$ the sum of rational functions
\begin{equation}
\label{eq:AjtAis}
\sum_{b=0}^{|t|-1}\frac{_iA_s(F)\cdot t^{-b}}{L_t^j\cdot t^{-b}}(x)
\end{equation}
is actually a polynomial.  There are two cases to consider here:

{\it Case 1:  $\prod_{b=0}^{|s|-1}x\cdot t^b(\ell_s)\in \langle x(\ell_t)\rangle$.}  In this case, we must have $L_s=L_t$ and either $\langle s\rangle \subseteq\langle t\rangle$ or $\langle t\rangle\subseteq\langle s\rangle$.  In the latter case, ${_iA_s}$ is $t$-invariant hence the sum in Equation \eqref{eq:AjtAis} is equal to zero.  In the former case, suppose we must have $s=t^a$ for some $a> 1$.  Expanding the ${_iA_s}(F)$ term in Equation \eqref{eq:AjtAis}, we get
\begin{align*}
\sum_{b=0}^{|t|-1}\frac{{_iA_s}(F)\cdot t^{-b}}{L_t^j\cdot t^{-b}}(x)= & \sum_{b=0}^{|t|-1}\sum_{c=0}^{|s|-1}\frac{F(xt^bs^c)}{L_s^i(xt^bs^c)\cdot L_t^j(xt^b)}\\
= & \sum_{b=0}^{|t|-1}\sum_{c=0}^{|t^a|-1}\frac{F(xt^bt^{ac})}{L_t^i(xt^bt^{ac})\cdot L_t^j(xt^b)}\\
= & \sum_{b=0}^{|t|-1}\sum_{c=0}^{|t^a|-1}\frac{F(xt^bt^{ac})\cdot \lambda_t^{acj}}{L_t^i(xt^bt^{ac})\cdot L_t^j(xt^bt^{ac})}\\
= & \sum_{b=0}^{|t|-1}\sum_{c=0}^{|t^a|-1}\frac{F(xt^bt^{ac})\cdot \lambda_t^{acj}}{L_t^{i+j}(xt^bt^{ac})}\\
& \left(\text{note that the sum $\sum_{b=0}^{|t|-1}\frac{F(xt^{b+ac})}{L_t^{i+j}(xt^{b+ac})}$ is independent of $c$}\right)\\
= & \sum_{b=0}^{|t|-1}\frac{F(xt^{b+ac})}{L_t^{i+j}(xt^{b+ac})}\left(\sum_{c=0}^{|t^a|-1} \lambda_t^{acj}\right)\\
= & 0.
\end{align*}
Hence in either subcase $_iA_s(F)(x)\in R$ in this case.

{\it Case 2:  $\prod_{b=0}^{|t|-1}x\cdot t^b(\ell_s)\notin\langle x(\ell_t)\rangle$.}  In this case, it will suffice to show that the product 
$$\prod_{a=0}^{|t|-1}L_s^i(xt^a)\cdot\left(\sum_{b=0}^{|t|-1}\frac{{_iA_s}(F)(xt^b)}{L_t^j(xt^b)}\right)$$
is in $R$.  Expanding we get
\begin{align*}
\prod_{a=0}^{|t|-1}L_s^i(xt^a)\cdot\left(\sum_{b=0}^{|t|-1}\frac{{_iA_s}(F)(xt^b)}{L_t^j(xt^b)}\right)= & \sum_{b=0}^{|t|-1}\sum_{c=0}^{|s|-1}\frac{\lambda_s^{-ic}\prod_{a\neq b}L^i_s(xt^a)\cdot F(xt^bs^c)}{L_t^j(xt^b)}\\
= & \sum_{c=0}^{|s|-1}\lambda_s^{-ic}\left(\sum_{b=0}^{|t|-1}\frac{\prod_{a\neq b}L^i_s(xt^a)\cdot (F\cdot s^c)(xt^b)}{L_t^j(xt^b)}\right)\\
= & \sum_{c=0}^{|s|-1}\lambda_s^{-ic}\left(\sum_{b=0}^{|t|-1}\frac{\prod_{a\neq 0}(L^i_s\cdot t^a)(xt^b)\cdot (F\cdot s^c)(xt^b)}{L_t^j(xt^b)}\right)\\
= & \sum_{c=0}^{|s|-1}\lambda_s^{-ic}\left({_jA_t}\left(\prod_{a\neq 0}(L^i_s\cdot t^a)\cdot (F\cdot s^{-c})\right)(x)\right)
\end{align*}
which is in $R$ because $\mathcal{H}_W$ is closed under multiplication.

Thus in every case, if $F\in\mathcal{H}_W$ then so is $_iA_s(F)$ for each pseudo-reflection $s\in s(W)$ and each integer $i\leq |s|-1$.
\end{proof} 

It turns out that the operators ${_iA_s}\colon\mathcal{H}_W\rightarrow\mathcal{H}_W(-i)$ are equivariant analogues of the so-called \emph{generalized $\Delta$-operators} ${_i\Delta_s}\colon R\rightarrow R(-i)$ introduced by Smith and McDaniel \cite{SM}.  We review this construction now.

For each pseudo-reflection $s\in s(W)$ and any integer $i\in\Z$, we define the operator ${_i\Delta_s}\colon R\rightarrow \operatorname{Quot}(R)(-i)$ as follows:  For any homogeneous polynomial $f\in R$, define 
\begin{equation}
\label{eq:gDo}
{_i\Delta_s}(f)=\sum_{j=0}^{|s|-1}\frac{s^j(f)}{s^j\left(\ell^i_s\right)}=\frac{1}{\ell_s^i}\sum_{j=0}^{|s|-1}\lambda_s^{-ij}\cdot s^j(f).
\end{equation} 
\begin{lemma}
\label{lem:SM1}
For any $s\in s(W)$ and any integer $i\leq |s|-1$ we have 
$${_i\Delta_s}(f)\in R \ \ \forall \ f\in R.$$
\end{lemma}  
\begin{proof}
We refer the reader to the paper \cite{SM} for details.
\end{proof}
We call the operators ${_i\Delta_s}\colon R\rightarrow R(-i)$ \emph{generalized $\Delta$-operators} for $s\in s(W)$ and $i\leq |s|-1$.

\begin{lemma}
\label{lem:ComDi}
The following diagram commutes:
$$\xymatrix{R\otimes_{R^W}R\ar[r]^-{\mu}\ar[d]_-{1\otimes {_i\Delta_s}} & \Maps(W,R)\ar[d]^-{{_iA_s}}\\
R\otimes_{R^W}R(-i)\ar[r]_-{\mu} & \Maps(W,R)(-i).\\}$$
\end{lemma}
\begin{proof}
It suffices to show that 
$$\left({_iA_s}\circ\mu\right)(f\otimes g)(x)=\left(\mu\circ\left(1\otimes{_i\Delta_s}\right)\right)(f\otimes g)(x)$$ for every simple tensor $f\otimes g\in R\otimes_{R^W}R$ and every $x\in W$.  On the LHS we have
\begin{align*}
\left({_iA_s}\circ\mu\right)(f\otimes g)(x)= & \sum_{j=0}^{|s|-1}\frac{\left(\mu(f\otimes g)\right)(x\cdot s^j)}{L_s^i(x\cdot s^j)}\\
= & \sum_{j=0}^{|s|-1}\frac{f\cdot (x\cdot s^j)(g)}{L_s^i(x\cdot s^j)}\\
= & f\cdot \sum_{j=0}^{|s|-1}\frac{x(s^j(g))}{x\left(s^j\left(\ell_s^i\right)\right)}\\
= & f\cdot x\left(\sum_{j=0}^{|s|-1}\frac{s^j(g)}{s^j(\ell_s^i)}\right),
\end{align*}
whereas on the RHS we have 
\begin{align*}
\left(\mu\circ\left(1\otimes{_i\Delta_s}\right)\right)(f\otimes g)(x)= & \mu\left(f\otimes{_i\Delta_s}(g)\right)(x)\\
= & f\cdot x\left({_i\Delta_s}(g)\right)\\
= & f\cdot x\left(\sum_{j=0}^{|s|-1}\frac{s^j(g)}{s^j(\ell^i_s)}\right).
\end{align*}
\end{proof}

Note that an immediate consequence of Lemma \ref{lem:ComDi} is that the image of the localization map $\mu\left(R\otimes_{R^W}R\right)$ is contained in our set $\mathcal{H}_W$.

Let us pause for a moment and take stock.  We know by Proposition \ref{prop:Injective} that the localization map $\mu\colon R\otimes_{R^W}R\rightarrow\bigoplus_{x\in W}R$ is injective.  As we pointed out above, Lemma \ref{lem:ComDi} implies that the image is contained in the subset $\mathcal{H}_W$.  Moreover $\mu$ is $W$-equivariant, which implies that the image $\mu\left(R\otimes_{R^W}R\right)\subset\mathcal{H}_W$ is a $W$-invariant subspace.  Thus the quotient of graded vector spaces $\mathcal{H}_W/\mu\left(R\otimes_{R^W}R\right)$ is a graded vector space itself which also carries a $W$-action.  We would like to show that this quotient space is zero, and hence that the image of $\mu$ is equal to $\mathcal{H}_W$.  We will use our operators $_iA_s\colon \mathcal{H}_W\rightarrow\mathcal{H}_W(-i)$.
\begin{proposition}
\label{prop:LocalSurj}
The localization map $\mu\colon R\otimes_{R^W}R\rightarrow\mathcal{H}_W$ is surjective.
\end{proposition}
\begin{proof}
Suppose by way of contradiction that $\mathcal{H}_W/\mu(R\otimes_{R^W}R)$ is not the zero space.  Then since it is non-negatively graded, we may choose a homogeneous nonzero element $\bar{F}\in\mathcal{H}_W/\mu(R\otimes_{R^W}R)$ of smallest degree.  Let $F\in\mathcal{H}_W$ be a homogeneous representative of $\bar{F}$.  For each pseudo-reflection $s\in s(W)$ and for each integer $1\leq i\leq |s|-1$, it then follows from the minimality of the degree of $F$ that the polynomial ${_iA_s}(F)$, being of degree strictly less than that of $F$, must lie in the image $\mu(R\otimes_{R^W}R)\subset\mathcal{H}_W$.  But now by Lemma \ref{lem:isDecomp}, we deduce that the difference 
$$F-\frac{1}{|s|}{_0A_s}(F)$$
must also lie in the image $\mu(R\otimes_{R^W}R)$.  But this means that in the quotient space $\mathcal{H}_W/\mu(R\otimes_{R^W}R)$, which carries a $W$-action, $\bar{F}$ is $s$-invariant since ${_0A_s}(F)$ is.  But this holds for all pseudo-reflections, and since $W$ is generated by pseudo-reflections, we see that $\bar{F}$ is actually $W$-invariant.  Upstairs in $\mathcal{H}_W$ we can reformulate this by writing
\begin{equation}
\label{eq:BarF}
F-F\cdot w\in\mu(R\otimes_{R^W}R) \ \ \forall w\in W.
\end{equation}
Summing Equation \eqref{eq:BarF} over all $w\in W$ and dividing by $|W|$ we get that 
$$F-\frac{1}{|W|}\sum_{w\in W}F\cdot w\in\mu(R\otimes_{R^W}R).$$
The problem with the last displayed equation is that the map 
$$\frac{1}{|W|}\sum_{w\in W}F\cdot w\colon W\rightarrow R$$ 
is $W$-invariant hence it must be a constant map, which is also in the image $\mu(R\otimes_{R^W}R)$.  Hence $F$ is forced to lie in the image as well, which is the desired contradiction.  
\end{proof}
We have thus proved Theorem \ref{thm:Main}.

\section{Equivariant Cohomology of Linear Hypergraphs}
By a \emph{hypergraph} we mean a pair $\Gamma=(\mathcal{V},\mathcal{E})$ consisting of a (finite) vertex set $\mathcal{V}$, and collection of subsets $\mathcal{E}$ of $\mathcal{V}$ called the \emph{hyperedges}.  For $e\in\mathcal{E}$ we write $|e|$ for the number of vertices in $e$, and we write $\mathcal{V}_e\subset\mathcal{V}$ for the underlying vertex set.  For $p\in\mathcal{V}$ we write $\mathcal{E}_p\subset\mathcal{E}$ to mean the subset of hyperedges containing the vertex $p$.

An \emph{axial function} on $\Gamma$ is a function $\alpha\colon \mathcal{E}\rightarrow \P(V^*)$ which assigns a linear subspace in $V^*$ to each edge.  Given $(\Gamma,\alpha)$, a \emph{generator class} for a given hyperedge $e\in \mathcal{E}$ is an injective map $\tau_e\colon\mathcal{V}_e\rightarrow \alpha(e)$ where $\alpha(e)$ is the linear subspace assigned to $e$ by $\alpha$.  Denote by $\tau$ the collection of maps $\left\{\tau_e\right\}_{e\in\mathcal{E}}$.  The triple $(\Gamma,\alpha,\tau)$ is what we will refer to as a \emph{linear hypergraph}.  

We define the equivariant cohomology of a fixed hyperedge $e\in\mathcal{E}$ as the subset of $\operatorname{Maps}(\mathcal{V}_e,R)$ given by 
\begin{equation}
\label{eq:HyperEdge}
H(e,\tau_e)\coloneqq \left\{G\colon \mathcal{V}_e\rightarrow R\left| G=\sum_{i=0}^{|e|-1}g_i\cdot\tau_e^i, \ \ \text{for some constants $g_i\in R$}\right.\right\}.
\end{equation}  
For every hyperedge $e\in \mathcal{E}$ there is a natural restriction map $\rho_e\colon \operatorname{Maps}(\mathcal{V},R)\rightarrow\operatorname{Maps}(\mathcal{V}_e,R)$.  Taking Equation \eqref{eq:HyperEdge} into account, we define the \emph{equivariant cohomology} of the linear hypergraph $(\Gamma,\alpha,\tau)$ as the set of $\operatorname{Maps}(\mathcal{V},R)$ given by 
\begin{equation}
\label{eq:HyperGraph}
H(\Gamma,\alpha,\tau)\coloneqq \left\{F\colon \mathcal{V}\rightarrow R\left| \rho_e(F)\in H(e,\tau_e), \ \ \text{for each $e\in\mathcal{E}$}\right.\right\}.
\end{equation}
\begin{remark}
Guillemin and Zara \cite{GZ2} encountered linear hypergraphs as cross sections of a GKM 1-skeleton, analogous to the reduced spaces of a symplectic manifold with a Hamiltonian circle action.  In their paper they defined the equivariant  
cohomology of their linear hypergraphs using an analogue of the Kirwan map.  On the other hand, their Theorem 7.1 states that their definition of equivariant cohomology of linear hypergraph actually agrees with the one given in Equation \eqref{eq:HyperGraph}.
\end{remark}

Let $W\subset\GL(V)$ be a finite psuedo-reflection group as above.  Define a hypergraph $\Gamma_W=(\mathcal{V}_W,\mathcal{E}_W)$ as follows.  Set the vertex set $\mathcal{V}_W=W$.  Denote by $\mathfrak{o}_s(x)$ the \emph{right $s$-orbit containing $x$}, i.e.
$$\mathfrak{o}_s(x)=\left\{x,x\cdot s,\cdots x\cdot s^{|s|-1}\right\}.$$
We then define the hyperedge set $\mathcal{E}_W$ to be the set of all right $s$-orbits, i.e.
$$\mathcal{E}_W=\left\{\mathfrak{o}_s(x)\left|s\in s(W), \ x\in W\right.\right\}.$$

Define an axial function $\alpha\colon\mathcal{E}_W\rightarrow \P(V^*)$ by 
\begin{equation}
\label{eq:AxialW}
\alpha_W\left(\mathfrak{o}_{s}(x)\right)=\F\cdot x(\ell_{s}),
\end{equation}
and for each hyperedge $\mathfrak{o}_{s}(x)$ define the generating class $\tau_{\mathfrak{o}_s(x)}=L_s\colon \mathfrak{o}_{s}(x)\rightarrow \alpha(\mathfrak{o}_{s_H}(x))$, i.e. $\tau_{\mathfrak{o}_s(x)}(x\cdot s^i)=\lambda_s^i\cdot x(\ell_s)$.  Set $\tau_W=\left\{\tau_e\right\}_{e\in\mathcal{E}_W}$.  Then the triple $(\Gamma_W,\alpha_W,\tau_W)$ is the linear hypergraph associated to $W$.

\begin{proposition}
\label{prop:Equality}
The sets $\mathcal{H}_W$ and $H(\Gamma_W,\alpha_W,\tau_W)$ coincide.
\end{proposition}
\begin{proof}
The condition that a map $F\colon W\rightarrow R$ be in $H(\Gamma_W,\alpha_W,\tau_W)$ is that for every hyperedge $e=\mathfrak{o}_s(x)$, the restriction $\rho_e(F)$ be in $H(e,\tau_e)$ where $\tau_e=L_s$, or, equivalently that 
$$\rho_e(F)=\sum_{i=0}^{|e|-1}g_i\cdot\tau_e^i,$$
for some constants $g_i\in R$.  On the other hand, recall Lemma \ref{lem:isDecomp} which says that if $F\in \mathcal{H}_W$ then 
$$F=\frac{1}{|s|}\sum_{i=0}^{|s|-1}{_iA_s}(F)\cdot L_s^i.$$
If follows that $\mathcal{H}_W\subseteq H(\Gamma_W,\alpha_W,\tau_W)$.

Conversely, suppose that $F\in H(\Gamma_W,\alpha_W,\tau_W)$.  Then for each hyperedge $e\in\mathcal{E}_W$ we have $\rho_e(F)=\sum_{i=0}^{|e|-1}f_i\tau_e^i$ for some constants $f_i\in R$.  On the other hand, Guillemin and Zara have shown that for any map $G\colon\mathcal{V}_e\rightarrow R$, the condition that $G=\sum_{i=0}^{|e|-1}g_i\tau_e^i$ for some constants $g_i\in R$ is equivalent to the condition that the sum of rational functions 
\begin{equation}
\label{eq:Integral}
\int_eG\cdot\tau_e^k\coloneqq \sum_{p\in \mathcal{V}_e}\frac{G(p)\cdot\tau_e(p)^k}{\prod_{\substack{q\in\mathcal{V}_e\\ q\neq p\\}}\left(\tau_e(p)-\tau_e(q)\right)}
\end{equation}
is a polynomial for each $k\geq 0$, c.f. \cite[Lemma 4.1]{GZ2}.  Taking $e=\mathfrak{o}_s(x)$ and $\tau_e=L_s\colon \mathfrak{o}_s(x)\rightarrow R$, and $G=\rho_e(F)$, Equation \eqref{eq:Integral} becomes
\begin{align*}
\int_e\rho_e(F)\cdot \tau_e^k= & \sum_{p\in\mathcal{V}_e}\frac{\rho_e(F)(p)\cdot\tau_e^k(p)}{\prod_{\substack{q\in\mathcal{V}_e\\ q\neq p\\}}\left(\tau_e(p)-\tau_e(q)\right)}\\
= & \sum_{j=0}^{|s|-1}\frac{F(x\cdot s^j)\cdot L_s^k(x\cdot s^j)}{\prod_{a\neq j}\left(L_s(x\cdot s^j)-L_s(x\cdot s^a)\right)}\\
= & \sum_{j=0}^{|s|-1}\frac{F(x\cdot s^j)\cdot \lambda_s^{kj}\cdot x(\ell_s)^k}{x(\ell_s)^{|s|-1}\cdot \prod_{a\neq j}\left(\lambda_s^j-\lambda_s^a\right)}\\ 
= & \sum_{j=0}^{|s|-1}\frac{F(x\cdot s^j)\cdot \lambda_s^{kj}\cdot x(\ell_s)^k}{x(\ell_s)^{|s|-1}\cdot \left(\lambda_s^j\right)^{|s|-1}\prod_{a\neq j}\left(1-\frac{\lambda_s^a}{\lambda_s^j}\right)}\\
& \left(\text{note that $\prod_{a\neq j}\left(1-\frac{\lambda_s^a}{\lambda_s^j}\right)=|s|$}\right)\\
= & \frac{1}{|s|}\sum_{j=0}^{|s|-1}\frac{F(x\cdot s^j)}{x(\ell_s)^{|s|-1-k}\cdot \lambda_s^{j\cdot(|s|-1-k)}}\\
= & \frac{1}{|s|}\sum_{j=0}^{|s|-1}\frac{F(x\cdot s^j)}{L^{|s|-1-k}_s(x\cdot s^j)}\\
= & \frac{1}{|s|}{_{|s|-1-k}A_s}(F)(x)
\end{align*}
Thus if $F\in H(\Gamma_W,\alpha_W,\tau_W)$ then ${_{|s|-1-k}A_s}(F)(x)\in R$ for each $x\in W$, each $s\in s(W)$, and each $k\geq 0$, which implies that $F\in\mathcal{H}_W$ as well.
\end{proof}

\bibliographystyle{plain}
\bibliography{Smith}

\begin{thebibliography}{1}

\bibitem{GHZ}
V.~Guillemin, T.~Holm, and C.~Zara.
\newblock A {GKM} description of the equivariant cohomology ring of a
  homogeneous space.
\newblock {\em J. Algebraic Combin.}, 23(1):21--41, 2006.

\bibitem{GZ0}
V.~Guillemin and C.~Zara.
\newblock Equivariant de {R}ham theory and graphs.
\newblock {\em Asian J. Math.}, 3(1):49--76, 1999.
\newblock Sir Michael Atiyah: a great mathematician of the twentieth century.

\bibitem{GZ1}
V.~Guillemin and C.~Zara.
\newblock 1-skeleta, {B}etti numbers, and equivariant cohomology.
\newblock {\em Duke Math. J.}, 107(2):283--349, 2001.

\bibitem{GZ2}
Victor Guillemin and Catalin Zara.
\newblock The existence of generating families for the cohomology ring of a
  graph.
\newblock {\em Adv. Math.}, 174(1):115--153, 2003.

\bibitem{Mat}
Hideyuki {Matsumura}.
\newblock {\em {Commutative ring theory. Transl. from the Japanese by M.
  Reid.}}
\newblock 1986.

\bibitem{SM}
Chris McDaniel and Larry Smith.
\newblock A bott-samelson like embedding of coinvariants of reflection groups.
\newblock {\em in preparation}.

\bibitem{Smith}
Larry {Smith}.
\newblock {\em {Polynomial invariants of finite groups.}}
\newblock Wellesley, MA: A. K. Peters, 1995.

\bibitem{S1}
Wolfgang {Soergel}.
\newblock {The combinatorics of Harish-Chandra bimodules.}
\newblock {\em {J. Reine Angew. Math.}}, 429:49--74, 1992.

\bibitem{W0}
J~Watanabe.
\newblock Some remarks on cohen-macaulay rings with many zero divisors and an
  application.
\newblock {\em Journal of Algebra}, 39(1):1 -- 14, 1976.

\end{thebibliography}

\end{document}